\documentclass[11pt,a4paper,reqno]{amsart}
\usepackage{amsfonts}
\usepackage{graphicx}
\usepackage{amsmath}
\usepackage{amssymb}
\usepackage{txfonts}
\usepackage{mathrsfs}
\usepackage{color}
\usepackage{soul}
\usepackage{revsymb}
\usepackage{bbm}
\usepackage{upgreek}
\usepackage[colorlinks=true,linkcolor=blue, urlcolor=red, citecolor=blue]{hyperref}
\usepackage[T1]{fontenc}
\usepackage{amsfonts}
\usepackage{epsfig}
\usepackage{caption}
\usepackage{graphicx}
\usepackage{amsmath}
\usepackage{amssymb}
\usepackage{color}
\usepackage{mathrsfs}
\usepackage{cleveref}
\DeclareMathAlphabet{\mathpzc}{OT1}{pzc}{m}{it}

\newcounter{main}

\newtheorem{theorem}{Theorem}[section]
\newtheorem{proposition}[theorem]{Proposition}
\newtheorem{lemma}[theorem]{Lemma}
\newtheorem{corollary}[theorem]{Corollary}

\newtheorem{remark}{Remark}[section]

\newtheorem{maintheorem}{Theorem}

\newtheorem{maincorollary}{Corollary}

\newtheorem{claim}{Claim}[section]

\newcounter{example}
{{\stepcounter{example}}{\flushleft {\bf Example \arabic{example}:}}}%
{\par}

\DeclareMathOperator{\GL}{GL}
\DeclareMathOperator{\SL}{SL}

\def \dim{{\rm dim}}

\renewcommand{\varepsilon}{\epsilon}

\newcommand{\R}{\mathbb{R}}
\newcommand{\w}{\omega}

\title[Linear differential equations with $L^p$-variation on the parameters]{Simple Lyapunov spectrum for linear homogeneous differential equations with $L^p$ parameters}

\author[D. Amaro]{Dinis Amaro}
\address{Centro de Matem\'atica e Aplica\c{c}\~oes (CMA-UBI), Universidade da Beira
	Interior, Rua Marqu\^es d'\'Avila e Bolama, 6201-001, Covilh\~a, Portugal.}
\email{dinis.amaro@ubi.pt}

\author[M. Bessa]{M\'{a}rio Bessa}
\email{bessa@ubi.pt}
\author[H. Vilarinho]{Helder Vilarinho}
\email{helder@ubi.pt}

\date{\today}

\begin{document}
	
	\begin{abstract}
	 In the present paper we prove that densely, with respect to an $L^p$-like topology, the Lyapunov exponents associated to linear con\-ti\-nuous-time cocycles $\Phi:\mathbb{R}\times M\to \GL(2,\mathbb{R})$ induced by second order linear homogeneous differential equations
	 	  $\ddot x+\alpha(\varphi^t(\omega))\dot x+\beta(\varphi^t(\omega))x=0$ 
are almost everywhere distinct. The coefficients $\alpha,\beta$ evolve along the $\varphi^t$-orbit for $\omega\in M$ and $\varphi^t: M\to M$ is an ergodic flow defined on a probability space. We also obtain the corresponding version for the frictionless equation $\ddot x+\beta(\varphi^t(\omega))x=0$  and for a Schr\"odinger equation 
$\ddot x+(E-Q(\varphi^t(\omega)))x=0$, inducing a cocycle $\Phi:\mathbb{R}\times M\to \SL(2,\mathbb{R})$.
	\end{abstract}
	
	\maketitle
	
	\noindent
	
	\textbf{Keywords:} Linear cocycles; Linear differential systems; Multiplicative ergodic theorem; Lyapunov exponents; second order linear homogeneous differential equations.
	
	\noindent
	
	\textbf{\textup{2010} Mathematics Subject Classification:} Primary: 34D08,  37H15,
	Secondary: 34A30, 37A20.

	\section{Introduction}
	
	\subsection{Non-autonomous linear differential equations}
	The behaviour of the Lyapunov exponents which are determined by the asymptotic growth of the expression $\log\|\Phi^t_A\|^{1/t}$ where $\Phi^t_A$ is a matricial solution of the autonomous differential equation $\dot U(t)=A\cdot U(t)$ and $A$ is a square matrix of the same order as $U(t)$, is a simple exercise of linear algebra. Standard linear algebraic computations allows us to determine the Lyapunov spectrum which is defined by the Lyapunov exponents  and its eigendirections. The dynamics of a perturbed system like $\dot U(t)=B\cdot U(t)$, where $B$ is a perturbation of $A$, is a problem that is well understood (see e.g. \cite{Ka}). A much more complicated and interesting situation was considered in the pioneering works of Lyapunov and intended to consider the non-autonomous case $\dot U(t)=A(t)\cdot U(t)$, where $A$ is a matrix depending continuously on $t$.  Not only the asymptotic demeanor of $\log\|\Phi^t_A\|^{1/t}$ as well as its stability proves to be a substantially more difficult issue. A standard way of looking to non-autonomous linear differential equations is to consider the language of linear cocycles (see \S\ref{LC} for full details) where being non-autonomous  is captured by a labelling through an orbit of a given flow $\varphi^t$ on a certain phase space.
	
	\subsection{The quest for positive Lyapunov exponents}
	A positive (or negative) Lyapunov exponent gives us the average exponential rate of divergence (or convergence) of two neighboring trajectories whereas zero exponents give us the absence of any kind of exponential behavior. Pesin's theory guarantee a strong stable/unstable manifold theory in the presence of non-zero Lyapunov exponents. These geometric tools underlie much of the central results in today's dynamical systems. Consequently, there is no doubt that detecting non-zero Lyapunov exponents is an important question in dynamics an issue dating back to the late sixtiees and the work of Millionshchikov \cite{Mi}. It the early eightees Cornelis and Wojtkowski \cite{CW}, and Ledrappier \cite{Le} obtained criteria for the positivity of the Lyapunov exponents and in the nineties Knill \cite{Knill} and Nerurkar \cite{N} proved that non-zero Lyapunov exponents are a $C^0$-dense phenomena for certain cocycles. In the late nineties Arnold and Cong \cite{AC0} proved the $L^p$-denseness of positive Lyapunov exponents and their strategy was widespread in \cite{BVi} by two of the authors. Using Moser-type methods based on the concept of rotation number allowed Fabbri and Johnson to obtain abundance of positive Lyapunov exponents for linear differential systems evolving on $\text{SL}(2,\mathbb{R})$ and based on a translation on the torus (see \cite{F,FJ1,FJ2} and also the work with Zampogni \cite{FJZ}). Clearly, finding a po\-sitive Lyapunov exponent in $\text{SL}(2,\mathbb{R})$ immediately enable us to obtain a negative Lyapunov exponent and thus the simplicity of the Lyapunov spectrum (i.e. all Lyapunov exponents are different). Several results on the positivity of Lyapunov exponents established in the last ten years or so bring up different new approaches \cite{BGV,BVi,Vi,DK,BBCMVX,X}. As a paradigmatic example we recall \cite{Av} where Avila obtained abundance of simple spectrum, on a quite large scope of topologies and on the two dimensional case.	
		
\subsection{Asymptotic behaviour of second order linear homogeneous differential equations from Lyapunov's viewpoint}
	It has been known for almost two centuries that there are serious constraints when we try to apply analytic methods to integrate most functions. Indeed,  \emph{Liouville theory} (see e.g \cite{R}) explicitly des\-cribes what kind of problems can arise when solving differential equations. The qualitative theory of differential equations created by Poincar\'e and Lyapunov turn out to be a clever approach to deal with this setback. Here we intend to analyze the asymptotic behavior of the solutions of second order homogeneous linear differential equations of the form
	\begin{equation}\label{damp}
		\ddot{x}(t) +\alpha(\varphi^t(\omega))\dot{x}(t)+\beta(\varphi^t(\omega))x(t)=0,
	\end{equation}
	with coefficients $\alpha$ and $\beta$ displaying $L^p$ regularity, varying in time along the orbits of a flow $\varphi^t$ and allowing an $L^p$-small perturbation on the parameters. Namely, we will describe its Lyapunov spectrum taking into account the possibility of making a $L^p$-type perturbation on its coefficients. Instead of deal with a single equation we will consider infinite equations simultaneously as explained now: we consider a time-continuous cocycle based on an ergodic flow $\varphi^t\colon M \to M$ with respect to a probability measure in $M$ and with a dynamics on the fiber defined by a linear flow $\Phi^t_A$ which is solution of the linear variational equation $\dot{U}(\w,t)= A(\varphi^{t}(\omega))\cdot U(\w,t)$ with generator 
	\begin{equation}\label{damp2}
	\begin{array}{cccc}
A\colon &M & \longrightarrow & \mathbb{R}^{2\times2} \\& \w & \longmapsto &  \left(\begin{matrix}0&1\\ -\beta(\omega)& -\alpha(\omega)\end{matrix}\right)
\end{array}
	\end{equation}
	Differential equations like (\ref{damp}) appear in large scale in physics, engineering, complex biological systems and numerous applications of mathematics.  The quintessential example is the simple damped pendulum free from external forces where $\alpha$ and $\beta$ are functions depending on $\omega\in M$ evolving along a flow $\varphi^t\colon M\rightarrow M$ for $t\in\mathbb{R}$. When $\alpha$ and $\beta$ are first integrals (i.e. functions that are constant along the orbits of the flow $\varphi^t$) related with $\varphi^t$, then \eqref{damp} can be solved by simple algorithms of an elementary course on differential equations. When the parameters vary in time, explicit solutions could be hard to get. This is the case when the \emph{frictional force} $\alpha$ and the \emph{frequency of the oscillator} $\beta$ change over time which, we must admit, is the most plausible to happen in nature. Notice that genera\-tors like $A$ in \eqref{damp2} generate a particular class of solutions. Clearly, when $\alpha\not=0$ the solutions evolve on a subclass of the general linear group $\text{GL}(2,\mathbb{R})$ and when $\alpha=0$ the solutions evolve on a subclass of the special linear group $\text{SL}(2,\mathbb{R})$. Therefore, a specific study should be made taking into consideration that perturbations \emph{must} belong to our class and not to the wider class of generators of cocycles evolving in $\text{GL}(2,\mathbb{R})$ or even in $\text{SL}(2,\mathbb{R})$. Questions related to this particular class were treated in several works like e.g. \cite{ACE, AW, Be2, FL, Lei, ABV}.

	Fixing \emph{position} and \emph{momentum} $(x(0),\dot x(0))$ we intend to study the asymptotic behavior when $t\rightarrow\infty$ of the pair $(x(t),\dot x(t))$ namely asymptotic exponential growth rate given by the \emph{Lyapunov exponent}. In the present work and broadly speaking we intend to answer the following question:

\medskip

\begin{quote}
\emph{Is it possible to perturb the coefficients $\alpha$ and $\beta$, in an $L^p$-topology, in order to obtain two distinct Lyapunov exponents? }
\end{quote}

\medskip

Of course that, when considering the autonomous case in \eqref{damp2}, say $\alpha$ and $\beta$ not depending on $\w$ previous question is easily answered. Indeed, consider $A_\beta$ in \eqref{damp2} with $\alpha=0$, then $A_0$ has a solution with trivial Lyapunov spectrum (a single Lyapunov exponent equal to 0) but any $A_\beta$ with small $\beta\not=0$ will produce a solution with simple Lyapunov spectrum (two Lyapunov exponents equal to $\pm\sqrt\beta$). The difficulty increases significantly when we consider the non-autonomous case.

\medskip 

The precise concepts that allow an adequate formalisation to express the above question will be presented in Theorem~\ref{main.theorem} and Corollaries~\ref{main.corollary} and~\ref{main.corollary2}. 	
	
	\section{Definitions and statement of the results}
			
			\subsection{Linear cocycles}\label{LC} In this section we present some definitions that will be useful in the sequel. Let $(M,\mathcal{M},\mu)$ be a probability space and let $ \varphi\colon \R \times M \to M$ be a \emph{metric dynamical system} (or \emph{flow}) in the sense that  is a measurable map and
	\begin{enumerate}
		\item $\varphi^t \colon M \to M$ given by $\varphi^t (\w) = \varphi(t,\w)$ preserves the measure $\mu$ for all $t \in \R$;
		
		\item $\varphi^0 = \text{Id}_{M}$ and $\varphi^{t+s}=\varphi^t\circ\varphi^s$ for all $t,s \in\R$.
	\end{enumerate}
	 Unless stated otherwise we will consider along the text that the flow is ergodic in the usual sense that there exist no invariant sets except zero measure sets and
their complements. Let $\mathcal{B}(X)$ be the Borel $\sigma$-algebra of a topological space $X$. A (continuous-time) linear \emph{random dynamical system} (RDS) on $(\R^2,\mathcal B(\R^2))$, or a (continuous-time) \emph{linear cocycle}, over $\varphi$ is a $(\mathcal{B}(\R)\otimes \mathcal{M}/\mathcal{B}(\text{GL}(2, \mathbb{R}))$-measurable map 	
	\begin{equation*}
	\Phi:\R\times M \to \text{GL}(2, \mathbb{R})
	\end{equation*}
	such that the mappings $\Phi(t,\w)$ forms a cocycle over $\varphi$, i.e.,
		\begin{enumerate}
			\item $\Phi(0,\w)=\text{Id}$ for all $\w\in M$; 
			\item $\Phi(t+s,\w)=\Phi(t,{\varphi^{s}(\w)})\circ\Phi(s,\w)$, for all $s, t\in\R$ and $\w\in M$,
		\end{enumerate}
	and $t\mapsto \Phi(t,\w)$ is continuous for all $\w\in M$. We recall that having $\w\mapsto\Phi(t,\w)$ measurable for each $t\in\R$ and $t\mapsto \Phi(t,\w)$ continuous for all $\w\in M$ implies that $\Phi$ is measurable in the product measure space. These objects are also called \emph{linear differential systems} (LDS) in the literature.

	\subsection{Kinetic linear cocycles}\label{kinetic}
	
	We begin by considering as motivation the non-autonomous linear differential equation which describes a motion of the damped harmonic oscillator as the \emph{simple pendulum} along the path $(\varphi^{t}(\omega))_{t\in\mathbb{R}}$, with $\w\in M$ described by the flow $\varphi$.
		Let $K\subset \mathbb{R}^{2\times2}$ be the set of matrices $2\times 2$ of type
		\begin{equation}\label{KK}
		\left(\begin{matrix} 0&1\\b&a\end{matrix}\right)
	\end{equation}
with $a, b\in\R$. Denote by $\mathcal{G}$ the set of measurable applications $A: M\rightarrow \mathbb{R}^{2\times2}$  and by $\mathcal{K}\subset \mathcal{G}$ the set of \emph {kinetic} measurable applications $A: M\rightarrow K$. As usual we identify two applications on $\mathcal G$ that coincide on a $\mu$ full measure subset of $M$. Consider measurable maps $\alpha\colon M \to \R$ and $\beta\colon M \to \R$.
		Take the differential equation given in \eqref{damp}. Considering $y(t)= \dot{x}(t)$ we may rewrite \eqref{damp} as the following vectorial first order linear system
	\begin{eqnarray}\label{E1}
		\dot{X}= A(\varphi^{t}(\omega))\cdot X,
	\end{eqnarray}
	where $X=X(t)=(x(t),y(t))^T=(x(t),\dot x(t))^T$ and $A\in\mathcal{K}$ is given by \eqref{damp2}. 
		For all $1\leq p<\infty$ we define 
	\[\mathcal{G}^p=\left\{A\in\mathcal G\colon \int_M \|A\|^p d\mu<\infty\right\},
\]
where $\|\cdot\|$ denotes de standard Euclidean matrix norm. It is clear that for all $1\leq p<q<\infty$ we have $\mathcal G^q\subset \mathcal G^p$.
		 It follows from \cite[Thm. 2.2.2]{A} (see also Lemma 2.2.5 and Example 2.2.8 in this reference) that if $A\in \mathcal{G}^1$ then it generates a unique (up to indistinguishability)  linear RDS $\Phi_A$ satisfying
	\begin{equation}\label{eq:LDS}
		\Phi_{A}(t,\omega)=\text{Id}+\int_{0}^{t}A(\varphi^{s}(\omega))\cdot\Phi_{A}(s,\omega)\,ds.
	\end{equation}
		The solution $\Phi_A(t,\w)$ defined in \eqref{eq:LDS} is called \emph{the Carath\'eodory solution} or \emph{weak solution}. Given an initial condition $X(0)=v\in\mathbb{R}^2$, we say that $t\mapsto \Phi_A(t,\w)v$ solves or is a solution of \eqref{E1}, or that ~\eqref{E1} generates $\Phi_{A}(t,\w)$. Note  that $\Phi_A(0,\w)v=v$ for all $\w\in M$ and $v\in\mathbb{R}^{2}$. If the solution \eqref{eq:LDS} is differentiable in time (i.e. with respect to $t$) and satisfies for all $t$
	\begin{equation}\label{eq:LDS2}
		\frac{d}{dt}\Phi_A(t,\w)v=A(\varphi^t(\w))\cdot\Phi_A(t,\w)v\,\,\,\,\,\,\,\,\,\,\text{and}\,\,\,\,\,\,\,\,\,\, \Phi_A(0,\w) v=v,
	\end{equation}
	then it is called a \emph{classical solution} of ~\eqref{E1}. Of course that $t\mapsto \Phi_A(t,\w)v\,$ is continuous for all $\w$ and $v$. Due to \eqref{eq:LDS2} we call $A:M\to K$ a (kinetic) `infinitesimal generator' of $\Phi_A$. Sometimes, due to the relation between $A$ and $\Phi_A$, we refer to both $A$ and $\Phi_A$ as a kinetic linear cocyle/RDS/LDS. If \eqref{E1} has initial condition $X(0)=v$ then $\Phi_A(0,\w)v=v$ and $X(t)=\Phi_A(t,\w)v$.
			
Let $\mathcal{K}_0\subset \mathcal{K}$ stand for the \emph{traceless kinetic cocycles} derived from matrices as in~\eqref{KK} but with $a=0$. For $1\leq p< \infty$ set $\mathcal{K}^p=\mathcal{K}\cap \mathcal G^p$ and $\mathcal{K}^p_0=\mathcal{K}_0\cap \mathcal G^p\subset \mathcal{K}^p$.

		\bigskip

\subsection{The $L^p$ topology}
	We begin by defining an $L^p$-like topology generated by a metric that compares the infinitesimal generators on $\mathcal G$. Given $1\leq p < \infty$ and $A,B\in\mathcal{G}$ we 
	set 
			\begin{equation*}
		\hat\sigma_p(A,B):=\left\{\begin{array}{lll} \displaystyle\left(\int_{M} \|A(\omega)-B(\omega)\|^p\,d\mu(\w)\right)^{\frac1p}, \\    \infty\,\,\,\text{if the above integral does not exists,}         \\    \end{array}\right.
	\end{equation*}
 and define
	\begin{equation*}
		\sigma_p(A,B):=\left\{\begin{array}{lll}\frac{\hat\sigma_p(A,B)}{1+\hat\sigma_p(A,B)},\,&&\text{if}\,\, \hat\sigma_p(A,B)<\infty\\   1,\,&&\text{if}\,\, \hat\sigma_p(A,B)=\infty         \\    \end{array}\right..
	\end{equation*}
	Clearly, $\sigma_p$ is a distance in $\mathcal{G}$. It can be understood has  a version of the $L^p$-distance. Next topological content results were mainly proved in \cite{ABV}. The remaining statements follow straightforwardly.

\begin{proposition}\label{complete}
Consider $1\leq p <\infty$. Then:
\begin{enumerate}
\item[(i)] $\sigma_p(A,B)\leq \sigma_q(A,B)$ for all $1\leq p\leq q<\infty$ and all $A,B\in\mathcal{G}$.
\item[(ii)] If $A\in\mathcal{G}^1$ then $\sup_{0\leq t\leq 1}\log^{+}\|\Phi_A(t,\omega)^{\pm1}\|\in L^{1}(\mu)$.
\item[(iii)] If $A\in\mathcal{G}^p$ then for any $B\in\mathcal{G}$ satisfying $\sigma_p(A,B)<p$ we have $B\in\mathcal{G}^p$.
\item[(iv)] The sets $(\mathcal{K}^p,\sigma_p)$ and $(\mathcal{K}^p_0,\sigma_p)$ are closed, for all $1\leq p <\infty$.
\item[(v)] For all $1\leq p <\infty$, $(\mathcal{K}^p,\sigma_p)$ and $(\mathcal{K}^p_0,\sigma_p)$ are complete metric spaces and, therefore Baire spaces.
\end{enumerate}
\end{proposition}	

Next results are elementary in measure theory nevertheless we will use it often. They capture the whole idea of making huge perturbations on the uniform norm but small perturbations in the $\sigma_p$-distance as long the support is small in measure.

\begin{lemma}\label{simples}
Let $1\leq p <\infty$. Given $A\in \mathcal{G}^p$ and $\epsilon>0$ there exists $\delta>0$ such that if $\mathcal{F}\in \mathcal{M}$ and $\mu(\mathcal{F})<\delta$, then $\int_\mathcal{F}\|A(\w)\|^p\,d\mu(\w)<\epsilon$.
\end{lemma}
\begin{proof}
	The proof is made by contradiction. Suppose that exists $\epsilon>0$ and $\mathcal{F}_n\in \mathcal{M}$, for each $n\in\mathbb{N}$, such that $\mu(\mathcal{F}_n)<\frac{1}{2^n}$ and 
	\begin{equation}\label{contra}
	\int_{\mathcal{F}_n}\|A(\w)\|^{p}\,d\mu(\w)\geq\epsilon.
	\end{equation}
	Letting $\mathcal{F}=\limsup_n\mathcal{F}_n$, by the Borel-Cantelli lemma $\mu(\mathcal{F})=0$, and so 
	\begin{equation}\label{contra2}
	\int_{\mathcal{F}}\|A(\w)\|^p\,d\mu(\w)=0.
	\end{equation}
	The following leads to a contradiction:
	\begin{eqnarray*}
	\epsilon&\overset{\eqref{contra}}{\leq}&\limsup\int_{\mathcal{F}_n} \|A(\w)\|^p\,d\mu(\w)=\limsup\int \|A(\w)\|^p\chi_{\mathcal{F}_n}(\w)\,d\mu(\w)\\
	&\overset{\star}{\leq}&\int \limsup\|A(\w)\|^p\chi_{\mathcal{F}_n}(\w)\,d\mu(\w)=\int \|A(\w)\|^p\chi_{\mathcal{F}}(\w)\,d\mu(\w)\\
	&=& \int_{\mathcal{F}} \|A(\w)\|^p\,d\mu(\w)\overset{\eqref{contra2}}{=}0,
	\end{eqnarray*}
	where in $\star$ we used the reverse Fatou lemma.
\end{proof}
		
\begin{corollary}\label{coro:lp}
Let $1\leq p <\infty$,  $A\in \mathcal{G}^p$ and $\epsilon>0$ be given. Consider $B\in \mathcal{G}^p$ such that $A(\w)\neq B(\w)$ if and only if $\w\in\mathcal F$ for some $\mathcal F\in\mathcal M$ (that is, $B$ only differs from $A$ in $\mathcal F$).  Then there exists $\delta>0$ such that if $\mu(\mathcal{F})<\delta$ we have $\sigma_p(A,B)<\epsilon$.
\end{corollary}
\begin{proof}
Is is enough to prove that $\hat\sigma_p(A,B)<\epsilon$. For that, apply Lemma~\ref{simples} for $(A-B)\in \mathcal G^p$ and $\epsilon^p$.
\end{proof}

	\subsection{Statement of Theorem~\ref{main.theorem} and a tour on its proof }
Let $1\leq p <\infty$ and $A\in\mathcal{K}^p$. Since $\mathcal{K}^p\subset\mathcal{K}^1\subset \mathcal{G}^1$, from Proposition~\ref{complete} the cocycle $\Phi_A$ satisfies the following \emph{integrability condition}
	\begin{equation}\label{eq IC}
	\sup\limits_{0\leq t\leq 1}\log^{+}\|\Phi_A(t,\omega)^{\pm1}\|\in L^{1}(\mu).
	\end{equation}	
	Hence, under condition \eqref{eq IC} Oseledets theorem (see e.g. \cite{O,A}) guarantees that for $\mu$ almost every $\omega\in M$, there exists a $\Phi_A$-invariant splitting, called \emph{Oseledets splitting}, of the fiber $\mathbb{R}^{2}_{\omega}=E^{1}_{\omega}\oplus E^{2}_{\omega}$ and real numbers $\lambda_{1}(A,\omega)\geq\lambda_{2}(A,\omega)$, called \emph{Lyapunov exponents}, such that:
	\begin{equation*}
		\lambda(A,\omega,v_i):=\underset{t\rightarrow{\pm{\infty}}}{\lim}\frac{1}{t}\log{\|\Phi_A(t,\omega) v_{i}\|=\lambda_{i}(A,\omega)},
	\end{equation*}
	for any $v_{i}\in{E^{i}_{\w}\setminus\{\vec{0}\}}$ and $i=1, 2$.  
		If the flow $\varphi^{t}$ is ergodic, then the Lyapunov exponents (and the dimensions of the associated subbundles) are constant $\mu$ almost everywhere, and we refer to them as $\lambda_1(A)$ and  $\lambda_2(A)$, with $\lambda_1(A)\geq\lambda_2(A)$. We say that $A$ (or $\Phi_A$) has \emph{one-point Lyapunov spectrum} or \emph{trivial Lyapunov spectrum} if for $\mu$ a.e. $\w\in M$, $\lambda_{1}(A,\omega)= \lambda_{2}(A,\omega)$. Otherwise we say  $A$ (or $\Phi_A$) has \emph{simple Lyapunov spectrum}. For details on these results see \cite{A} (in particular, Example 3.4.15). 
				
	\medskip
	
	We are now in conditions to state our main result that establishes the existence of a $\sigma_p$-dense subset of $\mathcal{K}^p$ displaying simple spectrum:
	
	\medskip
	
	\begin{maintheorem}\label{main.theorem}
		Let $\varphi^t:M\to M$ be ergodic. For any $1\leq p<\infty$, $A\in\mathcal{K}^p$ and $\epsilon>0$, there exists $B\in\mathcal{K}^p$ exhibiting simple Lyapunov spectrum satisfying $\sigma_p(A,B)<\epsilon$.
	\end{maintheorem}

	This result shows in particular that the $\sigma_p$-generic subset of $\mathcal{K}^p$ in which the trivial spectrum prevails, obtained in \cite{ABV}, can not contain $\sigma_p$-open sets. The stra\-te\-gy to prove that for each {kinetic} cocycle satisfying the integrability condition there is another kinetic cocycle, arbitrarily close with a simple spectrum, borrow some ideas of~\cite{AC0,BVi} where the authors obtained a similar result for the discrete time case and for more general cocycles. However, the context of continuous-time cocycles and the restriction to a very particular family of cocycles, such as the one we are considering in this paper, bring several difficulties that have no similarities in previous works. We have to face the situation that kinetic cocycles are rigid\footnote{\,The pertubative arguments in \cite{AC0,BVi} were easier to make because since $\dim\SL(2,\mathbb{R})=3$ three degrees of freedom were available. In our kinetic scenario we have to perform the same perturbations but with only a single degree of freedom.} and to obtain the desired perturbation we will make a step-by-step perturbation algorithm that we now describe:
\begin{itemize}
\item [(1)] We begin by coding $\varphi$ by a special flow to avoid overlaps and then consider a \textit{thin} time-1 flowbox $\mathcal{V}_R$ concatenated to an also \textit{thin} time-$1$ flowbox $\mathcal{V}_S$, so that o $\mathcal{V}_R\cup \mathcal{V}_S$ will be a time-$2$ flowbox;
\item [(2)] We \textit{cut} the original dynamics in $\mathcal{V}_R$ (respectively $\mathcal{V}_S$) and paste a simple constant traceless infinitesimal generator $R_{2\pi}$, whose solution basically rotates an angle $2\pi\eta$ in time-$\eta$. Outside $\mathcal{V}_R\cup \mathcal{V}_S$ we keep the same dynamic of $A$. By \emph{simple} we mean that we can easily obtain the identity by just doing a time-$1$ iteration. Call $A_0$ this new cocycle;
\item [(3)] Since $\mathcal{V}_R\cup \mathcal{V}_S$ is a thin flowbox, $A_0$ will be arbitrarily $\sigma^p$-near $A$. If $A_0$ has simple spectrum we are over, otherwise we prove Theorem~\ref{main.theorem} for $A_0$ instead of $A$;
\item [(4)] Inside $\mathcal{V}_R$ we \textit{cut} the dynamics of $A_0$ and \textit{paste} a tailor-made rotation $R$ such that for each $\w$ entering in $\mathcal{V}_R$ we rotate in time-1 a vector $v_\w$ into a fixed special direction  given by $v=(1,1)$. The vector $v_\w$ will be used to forcefully create an Oseledets direction so we can calculate the Lyapunov exponents. Here we rotate any angle by a small $\sigma_p$-perturbation since by (1) $\mathcal{V}_R$ is thin. 
A key observation is that the trace keeps unchanged, and that is the main motivation to the previous placement of $R_{2\pi}$ on $\mathcal V_R$. Call $B_0$ this new cocycle. If $B_0$ has simple spectrum we are over, otherwise we prove Theorem~\ref{main.theorem} for $B_0$ instead of $A_0$;
\item [(5)] Inside $\mathcal{V}_S$ we \textit{cut} the dynamics of $B_0$ and \textit{paste} a constant infinite\-si\-mal generator $S$ which stretch the vector $v$ in time-$1$ by a known magnitude $e$. No problem arises with the (eventually large) size of the uniform norm of the perturbation because the $\sigma^p$-distance is small due to the thickness of $\mathcal{V}_S$. Again the trace keeps unchanged. Call $B$ this new cocycle;
\item [(6)] Now we use ergodicity and compute the Lyapunov exponents of points who will inevitably have to return to $\mathcal{V}_R\cup \mathcal{V}_S$ infinitely many times;
\item [(7)] The stretch $S$ is a perturbation that is concerned with providing an expansion along an invariant direction. As it is difficult to find different kinetic cocycles which keep the same invariant directions here it becomes clear why we have chosen back there the identity after time $1$ (more precisely a rotation by $2\pi$) given by $R_{2\pi}$;
\item [(8)] Finally, the concern to keep the trace constant in (4) and (5) will bear fruit since if a perturbation increases a Lyapunov exponent and simultaneously the sum of the two Lyapunov exponents of the original cocycle and the perturbed one remains the same, then only one thing could have happened: the perturbed cocycle cannot have trivial spectrum but instead must display a Lyapunov exponent smaller than the Lyapunov exponent of the original cocycle.
\end{itemize}

The following table summarises the step-by-step construction from the linear differential systems $A$ to $B$:

\medskip

\begin{center}
\captionof{table}{Step-by-step description of the several perturbations}\label{table}
\begin{tabular}{c*{6}{c}r}
Cocycle              & $M\setminus(\mathcal{V}_R\cup \mathcal{V}_S)$ & $\mathcal{V}_R$ & $\mathcal{V}_S$  \\
\hline
$A$ & $A$ & $A$ & $A$   \\
$A_0$            & $A$ & $R_{2\pi}$ & $R_{2\pi}$   \\
$B_0$          & $A$ & $R$ & $R_{2\pi}$  \\
$B$     & $A$ & $R$ & $S$   \\
\end{tabular}
\end{center}

\medskip

We use an approach slightly different from the previous works \cite {AC0, BVi, AC01, CO}. Moreover, to avoid overlapping in the perturbations, we will encode the base flow through a special flow in a Kakutani Castle (as in \cite{AK,Rudo}).
On the other hand, to estimate the proximity of the perturbed cocycle to the original one, we also use a control over the measure of $\mathcal{V}_R\cup \mathcal{V}_S$ that support the two perturbations taking into account Corollary~\ref{coro:lp}.

It should be noted that, in addition to the difficulties inherent in the context of continuous-time cocycles, performing these perturbations (rotation and stretch) are not trivial, as we do not have the usual mechanisms like those that exist in the context in cocycles that evolve in $\text{GL}(2,\mathbb{R})$ or $\text{SL}(2,\mathbb{R})$, or, more generally, cocycles that satisfy the \emph{accessibility} condition (also recognized as \emph{twisting}) and \emph {saddle-conservative} (also known as \emph{pinching}), which allow the realization of these processes in a less demanding way, as, for example, in \cite{AB, AC0, B, BV2, BVi}. 

As our perturbations are all traceless we get from Theorem~\ref{main.theorem} that conservative kinetic cocycles have non-zero Lyapunov exponents $\sigma_p$-densely.

\begin{maincorollary}\label{main.corollary}
		Let $\varphi^t:M\to M$ be ergodic. For any $1\leq p<\infty$, $A\in\mathcal{K}^p_0$ and $\epsilon>0$, there exists $B\in\mathcal{K}^p_0$ exhibiting non-zero Lyapunov exponents satisfying $\sigma_p(A,B)<\epsilon$.
	\end{maincorollary}

	Finally, we present Corollary~\ref{main.corollary} with a somewhat different look, namely by considering the one-dimensional Schr\"odinger operator on $L^2(\mathbb{R})$ and with an $L^p$ potential $Q\colon M\rightarrow \mathbb{R}$ given by:
	\begin{equation}\label{Sch}
	\begin{array}{cccc}
H_\w\colon &L^2(\mathbb{R}) & \longrightarrow & L^2(\mathbb{R}) \\& \phi & \longmapsto &  \left[-\frac{d^2}{dt^2}+Q(\varphi^t(\w))\right]\phi
\end{array}
\end{equation}

	In particular we like to describe the Lyapunov spectrum of the time-independent Schr\"odinger equation
	\begin{equation}\label{Sch2}
	H_\w\phi=E\phi,
	\end{equation}
	where $E\in \mathbb{R}$ is a given energy. Putting together \eqref{Sch} and \eqref{Sch2} we deduce a kinetic cocycle as in \eqref{damp2} but with $\alpha(\w)=0$ and $\beta(\w)=E-Q(\w)$ for all $\w\in M$. We fix the energy $E$ and focus on the LDS
\begin{equation}\label{damp22}
	\begin{array}{cccc}
A_E\colon &M & \longrightarrow & \mathbb{R}^{2\times2} \\& \w & \longmapsto &  \left(\begin{matrix}0&1\\ -E+Q(\omega) & 0\end{matrix}\right)
\end{array}
\end{equation}
called one-dimensional Schr\"odinger LDS with potential $Q$. As a direct consequence of Corollary~\ref{main.corollary} we have:

\begin{maincorollary}\label{main.corollary2}
		Let $\varphi^t:M\to M$ be ergodic. Given $1\leq p < \infty$, $\epsilon>0$ and a one-dimensional Schr\"odinger LDS with a fixed energy $E$ as in \eqref{damp22} and with potential $Q$, there exists $\tilde Q$ such that the one-dimensional Schr\"odinger LDS with the same energy $E$ and potential $\tilde Q$ exhibits non-zero Lyapunov exponents and $\|\tilde Q-Q\|_{L^p}<\epsilon$.
	\end{maincorollary}

\section{On the perturbations}

	\subsection{Special flows}
Consider a measure space $\Sigma$, a map $\mathcal{T}\colon \Sigma\rightarrow{\Sigma}$, a $\mathcal{T}$-invariant probability measure $\tilde{\mu}$ defined in $\Sigma$ and a roof function
$h\colon \Sigma\rightarrow{\mathbb{R}^{+}}$ satisfying $h(\w)\geq H>0$, for some $H>0$ and all $\w\in{\Sigma}$, and $
 \int_{\Sigma}h(\w)d\tilde{\mu}(\w)<\infty
$. Define the space $M_{h}\subseteq{\Sigma\times{\mathbb{R}_+}}$ by
$$
M_h=\bigl\{(\w,t) \in \Sigma\times{\mathbb{R}_+}: 0 \leq t \leq h(\w) \bigr\}
$$
with the identification between the pairs $(\w,h(\w))$ and $(\mathcal{T}(\w),0)$. The semiflow defined
on $M_h$ by $S^s(\w,r)=(\mathcal{T}^{n}(\w),r+s-\sum_{i=0}^{n-1}h(\mathcal{T}^{i}(\w)))$,
where $n\in{\mathbb{N}}$ is uniquely defined by
$$
\sum_{i=0}^{n-1}h(\mathcal{T}^{i}(\w))\leq{r+s}<\sum_{i=0}^{n}h(\mathcal{T}^{i}(\w))
$$
is called a \emph{suspension semiflow}. If $\mathcal{T}$ is invertible then $(S^t)_t$ is a flow.
Furthermore, if $\ell$ denotes the one dimensional Lebesgue measure the measure $\mu=(\tilde\mu \times \ell)/\int h\, d\tilde\mu$ defined on $M_h$
by
$$
\int g \, d\mu= \frac{1}{\int h\, d\tilde\mu} \int \left( \int_0^{h(\w)} g(\w,t) dt \right)\, d\tilde\mu(\w),  \quad \forall g\in C^0(M_h)
$$
is a probability measure and it is invariant by the suspension semiflow $(S^t)_t$. Flows with such representation are called \emph{special flows} (or \emph{flows built under a function}) and are denoted by $(\varphi^t,\Sigma,\mathcal T, h)$. It is well-known (see \cite[Theorem 2]{Am}) that any ergodic flow is isomorphic to a special flow. 
Along this work we assume that the base flow is a special flow $(\varphi^t,\Sigma,\mathcal T, h)$ and, without any loss of generality, that $H>2$. To avoid overloading the notation we write $M$ instead of $M_h$. 

\subsection{Perturbations supported in time-$\tau$ flowboxes}
Take $A\in\mathcal{G}$ and a non-periodic orbit $\omega\in M$. We will consider a perturbation $B=B_{\w,\tau}$ of $A$ only along a segment of the orbit of $\omega$ with extremes $\w$ and $\varphi^{\tau}(\w)$ for $\tau>0$. Let $P\in\mathcal{G}$ be given and define $B\colon M\to \mathbb{R}^{2\times2}$ such that $B(\hat\w)=A(\hat\w)$ for all $\hat\w$ outside $\varphi^{[0,\tau]}(\w)=\{\varphi^{s}(\w): s\in[0, \tau]\}$ and $B(\hat\w)=P(\hat\w)$ otherwise. The map $B$ is called  a (local) \emph{perturbation of $A$ by $P$ supported on $\varphi^{[0,\tau]}(\w)$}. 
Given $\Sigma_0\subset \Sigma$ and $0\leq a<b$ we define the set
\begin{equation*}
\varphi^{[a,b]}(\Sigma_0)=\left\{\varphi^t(\w)\colon \w\in \Sigma_0,\, t\in[a,b]\right\}.
\end{equation*}
Given $A\in\mathcal{G}^1$, $P\in\mathcal{G}$, $\Sigma_0\subset\Sigma$ and $a>0$, we may extend the local perturbations of $A$ by $P$ to be supported on the flowbox $\varphi^{[a,b]}({\Sigma_0})$, with $0\leq a<b<H$, in the following way: for $\w\in \varphi^{[a,b]}({\Sigma_0})$ we project $\w$ in $\tilde\w\in\varphi^a(\Sigma_0)$ i.e. $\w=\varphi^r(\tilde\w)$, for some $0\leq r\leq b-a$, and let $B_{\tilde\w,b-a}$ be (local) perturbation of $A$ by $P=P_{\tilde\w}$ supported on $\varphi^{[0,b-a]}(\tilde\w)$ and define
\begin{equation*}\label{perturbacao global}
		B(\w):=\left\{\begin{array}{lll}A(\w),\,&&\text{if}\,\, \w\notin  \varphi^{[a,b]}({\Sigma_0})\\   B_{\tilde\w,b-a}(\w),\,&&\text{if}\,\, \w\in  \varphi^{[a,b]}({\Sigma_0})         \\    \end{array}\right..
	\end{equation*}
To distinguish the situations we refer for $B(\w)$ as a \emph{global perturbation of $A$ by $P$ supported in $\varphi^{[a,b]}(\Sigma_0)$}, where we always suppose that $P(\w)=P_{\tilde\w}(\w)$ for all $\w\in\varphi^{[a,b]}({\Sigma_0})$.

\subsection{Rotating and Stretching}

Next two results provide local and global arguments to rotate over prescribed directions under a small $\sigma_p$-perturbation. This will be used to generate a suitable invariant direction. The first one allows us to perform a uniform bounded kinetic perturbation in a local segment of orbit which rotates a given vector. The second one thickens Lemma~\ref{rot4cont} by broaden the rotation in a single orbit to rotations in a flowbox.

\begin{lemma}\label{rot4cont}
Given $\w\in M$, $u,v\in\mathbb{R}^2\setminus\{ 0\}$, $A\in\mathcal{K}^p$, there is $\upgamma\neq 0$, and a perturbation $B_{\w,1}\in\mathcal{K}^p$ of $A$ supported on $\varphi^{[0,1]}(\w)$ such that:
  \begin{enumerate}
  \item[(i)] $\|B_{\w,1}(\hat \w)\|\leq 4\pi^2$ for all $\hat\w$ on $\varphi^{[0,1]}(\w)$, and
  \item[(ii)] $\Phi_{B_{\w,1}}(1,\omega)u=\upgamma\, v$.
  \end{enumerate}

\end{lemma}
	
\begin{proof}
Let $\theta=\measuredangle(\mathbb{R}u,\mathbb{R}v)\in\, ]0,2\pi]$ measured clockwise. Set a constant infinitesimal generator $R\colon M\to \mathbb{R}^{2\times2}$ given by 
\begin{equation}\label{infgen}
R(\w)=R_{\theta}(\w)=\left(\begin{matrix} 0 & 1\\ -\theta^2 & 0\end{matrix}\right).
\end{equation} 
We consider the perturbation $B=B_{\w,1}\in\mathcal{K}^p$ of $A$ by $R$ supported on $\varphi^{[0,1]}(\w)$. 
The infinitesimal generator in \eqref{infgen} generates a linear differential system with fundamental classical solution \eqref{eq:LDS2} given, for all $\w\in M$ and $t\in\mathbb{R}$ by the `clockwise elliptical rotation' defined by:
\begin{equation}\label{rotsol}
\Phi_{R}(t,\w)=\left(\begin{matrix}\cos(\theta t)& \theta^{-1}\sin(\theta t)\\ -\theta\sin(\theta t) & \cos(\theta t)\end{matrix}\right),
\end{equation} 
and such that  $\Phi_{B}(1,\w)u=\Phi_{R}(1,\w)u=\upgamma v$, for some $\upgamma\neq0$ fulfilling (ii). 
\end{proof}


From Corollary~\ref{coro:lp} it follows that we may extend the local perturbation $B_{\w,1}$ given by the rotation $R_\theta(\w)$ as in Lemma~\ref{rot4cont}, to a global perturbation, tuned for each orbit segment, to obtain a new generator that is $\sigma_p$-close to the original, once we have a smaller measure of the  flowbox were the perturbation takes place. This is pointed in the next basic measure theoretic result which is an immediate consequence of Corollary~\ref{coro:lp}.

\begin{lemma}[Global]\label{rot4pert}
  For all $1\leq p <\infty,$ $A\in\mathcal{G}^p$, $a>0$ and $\varepsilon>0$, there exists a measurable set $\Sigma_0\subset \Sigma$ with $\tilde\mu(\Sigma_0)>0$ such that for any global perturbation $B\in\mathcal{G}^p$ of $A$ supported in the flowbox $\varphi^{[a,a+1]}(\Sigma_0)$, with $\|B(\varphi^t(\w))\|\leq 4\pi^2$ for all $\w\in \Sigma_0$ and $t\in[a,a+1]$, we have that $\sigma_p(A,B)<\varepsilon$.
\end{lemma}
Let us fix a suitable constant and traceless infinitesimal generator
\begin{equation}\label{ST}
S=\begin{pmatrix}0&1\\1 &0\end{pmatrix}.
\end{equation}
As $S$ has simple expression we integrate it obtaining:
\begin{equation}\label{ST2}
\Phi_S(t,\w)=e^{St}=\begin{pmatrix}\cosh t&\sinh t\\\sinh t & \cosh t \end{pmatrix}
\end{equation}
We notice that \eqref{ST2} has eigenvalues $\sigma_1^S=e^t$ and $\sigma_2^S=e^{-t}$ with associated eigenvectors $v_1^S=(1,1)$ and $v_2^S=(-1,1)$,  respectively. Observe that $E_1^S=\mathbb{R}\cdot v_1^S$ is a unstable direction and $E_2^S=\mathbb{R}\cdot v_2^S$ is a stable direction. 
\medskip

Next trivial remark will be of utmost importance in the sequel because it combines three main ingredients: \emph{invariance} of certain 1-dimensional directions, some \emph{expansiveness} along this direction and all this done in \emph{traceless} kinetic infinitesimal generators.

\begin{remark}[Invariance and stretch]\label{estica}
Considering $\theta=2\pi$ in \eqref{rotsol}, say $R_{2\pi}$, we get
\begin{equation}\label{ST3}
e\cdot \,v_1^S=e\cdot\Phi_{R_{2\pi}}(1,\w)\,v_1^S=\Phi_S(1,\w)\,v_1^S.
\end{equation}
\end{remark}

\section{Proof of Theorem~\ref{main.theorem}}

Let $A\in\mathcal{K}^p$, $1\leq p<\infty$ and $\epsilon>0$ be given. We assume that $\Phi_A$ has a single Lyapunov exponent $\lambda(A)$. The sequence of perturbations are summarized in Table \ref{table}.

\subsection{Defining $A_0$ (picking out good coordinates):} Let $\Sigma_0\subset\Sigma$ be as in Lemma~~\ref{rot4pert}.
For $r>0$ we assume that we have flowboxes defined by $\mathcal V_R:=\varphi^{[0,1]}(B_r)$ and $\mathcal V_S:=\varphi^{[1,2]}(B_r)$, where $B_r\in \Sigma_0$ is such that $0<\tilde \mu(B_r)\leq r$. Consider $A_0\in\mathcal{K}^1$ defined as:

 \begin{equation*}
		A_0(\w):=\left\{\begin{array}{ll} A(\w),& \text{if $\w\notin  \mathcal V_R\cup\mathcal V_S$}\\ R_{2\pi},& \text{if $\w\in \mathcal V_R\cup\mathcal V_S$}   \end{array}\right..
	\end{equation*}
By Corollary~\ref{coro:lp} if $r$ is sufficiently small when compared with $\epsilon$ we get 
\begin{equation}\label{epsilon0}
	\sigma_p(A,A_0)<\frac{\epsilon}{3}.
	\end{equation}
If $\Phi_{A_0}$ has simple spectrum we are over. Otherwise, we prove the theorem for $A_0$ instead of $A$.

\subsection{Defining $B_0$ (rotating on $\mathcal V_R$):}
Set 
$$k(\w)=\inf_{t\geq0}\Bigl\{t\colon\varphi^{-t}(\w)\in \varphi^1(B_r)\Bigr\}.$$ We will define the a random vector field $g(\w)$. We start with the normalized image under the cocycle associated with $\Phi_{A_0}$ of  the vector $v=\frac{v_1^S}{\|v_1^S\|}=\left(\frac{\sqrt 2}{2},\frac{\sqrt 2}{2}\right)$:
	\[g(\omega)
	:=
	\begin{cases}
	v, & \text{if}\quad\omega\in \varphi^{1}(B_r)\\ \frac{\Phi_{A_0}(k(\w),\varphi^{-k(\w)}(\omega))v}{\|\Phi_{A_0}(k(\w),\varphi^{-k(\w)}(\omega))v\|},&\text{if}\quad\omega\notin (\mathcal{V}_R\setminus B_r)
	\end{cases}
	\]
	and set from now on $E(\w)=\text{span}\:\{g(\w)\}$.

	Let $B_0$ be a perturbation of $A_0$ supported in the flowbox $\mathcal V_R$ as in Lemma~\ref{rot4pert} such that for all $\w\in B_r$ we have $\Phi_{B_0}(1,\w)g(\w)=\kappa v$ for some $\kappa\in\R$, that is:
	 \begin{equation*}
		B_0(\w):=\left\{\begin{array}{ll} R(\w),& \text{if $\w\in \mathcal V_R$}\\ A_0(\w),&\text{otherwise }    \end{array}\right..
	\end{equation*}
	Observe that the rotation must be tuned for each $\w_0\in B_r$, in the sense that for $\w=\varphi^t(\w_0)\in\mathcal{V}_R$, with $0\leq t\leq 1$, we set $R(\w)=R_\theta(\w_0)$ with $\theta=\measuredangle(g(\w_0),v)$. In particular, for all $\w_0\in B_r$ we have $\Phi(1,\w_0)g(\w)=\kappa v$, for some $\kappa\in\R$. Moreover, $A_0$ and $B_0$ have the same trace. Indeed, $A_0=B_0$ outside $\mathcal V_R$ and in $\mathcal V_R$ we have $B_0=R$ and $A_0=R_{2\pi}$, which are both traceless (see \eqref{infgen}). 
	Therefore, by Liouville's formula for all $\w$ and $t\geq0$
	\begin{equation}\label{det1}
	\det \Phi_{B_0}(t,\w)=\det \Phi_{A_0}(t,\w).
	\end{equation}
	For $\w\in \mathcal{V}_R\setminus B_r$ define 
	\begin{equation}\label{gup}
	g(\w)=\frac{\Phi_{B_0}(k(\w),\varphi^{-k(\w)}(\omega))v}{\|\Phi_{B_0}(k(\w),\varphi^{-k(\w)}(\omega))v\|}.
	\end{equation}
Notice that for $\w\in B_r$, since $\Phi_{B_0}(1,\w)\mathbb{R}g(\w)=\mathbb{R} v$ we get
\begin{equation}\label{inv}
	\Phi_{B_0}(1,\w)\R g(\w)(\w)=\R g(\varphi^1(\w)).
	\end{equation}
Let $\tilde\w\in\varphi^1(B_r)$ and $\tau>0$ be such that $\varphi^t(\tilde\w)\notin \mathcal{V}_R$ for all $t\in]0,\tau[$. Then, for all $t\in[0,\tau]$ we have the $\Phi_{B_0}$-invariance of $g$:
\begin{equation}\label{inv2}
\Phi_{B_0}(t,\tilde\w)\mathbb{R}g(\tilde\w)=\Phi_{B_0}(t,\tilde\w)\mathbb{R}v=\Phi_{A_0}(t,\tilde\w)\mathbb{R}v=\mathbb{R} g(\varphi^{t}(\tilde\w)).
\end{equation}
If $\varphi^t(\tilde\w)\in \mathcal{V}_R$ for some $t\in]0,\tau[$ then considering $s>0$ such that $\varphi^s(\w)\in B_r$ we get:
\begin{eqnarray*}
\Phi_{B_0}(t,\tilde\w)\mathbb{R}g(\tilde\w)&=&\Phi_{B_0}(t-s,\varphi^s(\tilde\w))\Phi_{A_0}(s,\tilde\w)\mathbb{R}v\\
&=&\Phi_{B_0}(t-s,\varphi^s(\tilde\w))\mathbb{R}g(\varphi^{s}(\tilde\w))\\
&\overset{\eqref{gup}}{=}&\mathbb{R} g(\varphi^{t}(\tilde\w)).
\end{eqnarray*}
	Finally, \eqref{inv}, \eqref{inv2} and last equality gives that the vector field $g$ is $\Phi_{B_0}$-invariant. 
		
Again by Corollary~\ref{coro:lp} if $r$ is sufficiently small we get 
\begin{equation}\label{epsilon1}
	\sigma_p(A_0,B_0)<\frac{\epsilon}{3}.
	\end{equation}

If $\Phi_{B_0}$ has simple spectrum we are over. Otherwise, we prove the theorem for $B_0$ instead of $A_0$.

\subsection{Defining $B$ (stretching on $\mathcal V_S$):}

We define
 \begin{equation*}
		B(\w):=\left\{\begin{array}{ll} B_0(\w),& \text{if $\w\notin \mathcal V_S$}\\ S,&\text{if $\w\in \mathcal V_S$}     \end{array}\right..
	\end{equation*}
Observe that $B$ and $B_0$ have the same trace. Indeed, $B=B_0$ outside $\mathcal V_S$ and in $\mathcal V_S$ we have $B_0=R_{2\pi}$ which are both traceless (see \eqref{infgen} and \eqref{ST}). Therefore, by Liouville's formula and \eqref{det1} for all $\w$ and $t\geq0$
	\begin{equation}\label{det2}
	\det \Phi_B(t,\w)=\det \Phi_{B_0}(t,\w)=\det \Phi_{A_0}(t,\w).
	\end{equation}
	From Corollary~\ref{coro:lp}, once more, if $r$ is sufficiently small we get 
	\begin{equation}\label{epsilon2}
	\sigma_p(B_0,B)<\frac{\epsilon}{3}.
	\end{equation}
	Notice that the invariance of the direction $E(\w)$ under $\Phi_B$ fails when $\varphi^t(\w)$ enters $\mathcal{V}_S$. However, for $\tilde\w\in\varphi^1(B_r)$ we have  by~\eqref{ST3} and~\eqref{inv2} 
	\[
	\Phi_B(1,\tilde\w)\R g(\tilde\w)=\Phi_S(1,\tilde\w)\R v = \R v=\R\Phi_{R_{2\pi}}(1,\tilde\w)v=\Phi_{A_0}(1,\tilde\w)\R v= \R g(\varphi^1(\tilde\w))
	\]
	and so
	\begin{equation}\label{inv2}
\Phi_S(1,\tilde\w) E(\tilde\w) = E(\varphi^1(\tilde\w)),
	\end{equation}
	which will be enough for our purposes; see Figure~\ref{figure}.
\begin{figure}[h]
\begin{center}
  \includegraphics[scale=.15]{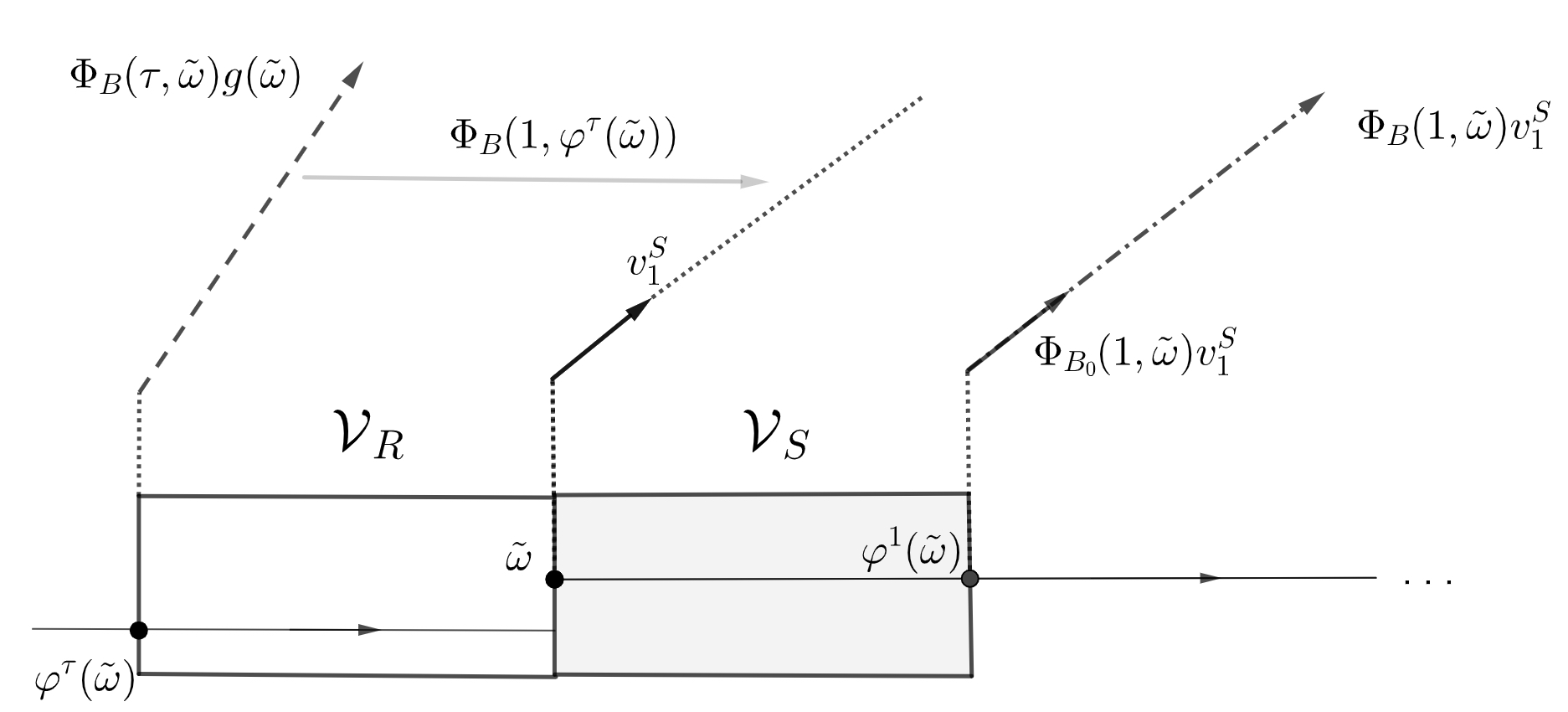}
  \caption{The traceless perturbation scheme with the invariant directions and the stretch effect.}
\label{figure}
\end{center}
\end{figure}

	Let $\lambda_1(B)\geq\lambda_2(B)$ be the Lyapunov exponents of $\Phi_B$. We assume that $\Phi_{B_0}$ has one-point spectrum, say $\lambda_1(B_0)=\lambda_2(B_0)=\lambda(B_0)$, because otherwise the theorem is proved. Let $\lambda(B_0)$ be the single Lyapunov exponent of $\Phi_{B_0}$. Hence we have $\lambda(B_0)=\lambda(B_0,\w,v_1^S)$ for a.e. $\w$.
By the Oseledets theorem we have
\begin{equation}\label{Cor1}
2\lambda(B_0)=\int\log\bigl|\det(\Phi_{B_0}(1, \w))\bigr|d\mu
\end{equation}
and
\begin{equation}\label{Cor2}
\lambda_1(B)+\lambda_2(B)=\int\log\bigl|	\det(\Phi_B(1,\omega))\bigr|d\mu.
\end{equation}
The two previous equalities together with \eqref{det2} allows us to conclude that \begin{equation}\label{2lbaB0iguallbda1+lbda2}
2\lambda(B_0)=\lambda_1(B)+\lambda_2(B)
\end{equation} and so, if we show that $\lambda_1(B)>\lambda(B_0)$ then we get $\lambda_1(B)>\lambda_2(B)$ and Theorem~\ref{main.theorem} is proved. 
Recall that the random vector field $g$ is invariant by $\Phi_{B_0}$ but in what $\Phi_B$ concerns, the invariance fails as the base dynamics enters $\mathcal V_S$. However, by~\eqref{inv2} the invariance is recovered in the moment the base dynamics is leaving $\mathcal V_S$.

For $\w \in M$ let us consider the real map $b_0(\cdot,\w)$ for all $t\in\R$ in such a way that
\begin{equation}\label{cocycle0}
b_0(t,\omega)g(\varphi^t(\w))=\Phi_{B_0}(t,\omega) g(\omega).
\end{equation}
\begin{claim}\label{cocycle}
The map $b_0(t,\omega)$ forms a cocycle over $\varphi^t$. 
\end{claim}
Indeed, since $\Phi_{B_0}(0,\omega)=Id$ for all $\w\in M$ we have $b_0(0,\omega)=1$ and for all $s,t$, evaluating $b_0(t+s,\w)$ at $g(\varphi^{t+s}(\w))$, we have 
\begin{eqnarray*}
b_0(t+s,\w)g(\varphi^{t+s}(\w))&\overset{\eqref{cocycle0}}{=}&\Phi_{B_0}(t+s,\omega) g(\omega)\\
&=&\Phi_{B_0}(t,\varphi^s(\w))\cdot \Phi_{B_0}(s,\w) g(\w)\\
&\overset{\eqref{cocycle0}}{=}&\Phi_{B_0}(t,\varphi^s(\w))\cdot b_0(s,\w) g(\varphi^s(\w))\\
&=&b_0(s,\w) \,\Phi_{B_0}(t,\varphi^s(\w))g(\varphi^s(\w))\\
&=&b_0(t,\varphi^s(\w))b_0(s,\w)g(\varphi^{t+s}(\w)),
\end{eqnarray*}
and so $b_0(t+s,\w)=b_0(t,\varphi^s(\w))b_0(s,\w)$. \\

Since the random vector field $g$ is not completely invariant by $\Phi_B$ we consider two distinct situations. Set $\varphi^{\{1,2\}}(B_r)=\varphi^{1}(B_r)\cup\varphi^2(B_r)$. For $\w \in M$ and $\tau\geq0$ such that $\varphi^t(\w)\notin\mathcal V_S\setminus\varphi^{\{1,2\}}(B_r)$, for all $0\leq t\leq \tau$, we consider
the real map $b(\cdot,\w)$ for all $t\in[0,\tau]$ in such a way that
\begin{equation}\label{cocycle1eumterco}
b(t,\omega)g(\varphi^t(\w))=\Phi_{B}(t,\omega) g(\omega)
\end{equation} and, for all $\w\in\varphi^1(B_r)$, we set $b(1,\w)\in\R$ in such a way that 
\begin{equation}\label{cocycle1emeio}
\Phi_{B}(1,\w)g(\w)=b(1,\omega)g(\varphi^1(\w)).
\end{equation}
If $\varphi^t(\w)\notin \mathcal V_S\setminus\varphi^{\{1,2\}}(B_r)$, for all $0\leq t\leq \tau$, we have $B(\varphi^t(\w))=B_0(\varphi^t(\w))$ and
\begin{equation}\label{cocycle1tresquartos}
b(t,\omega)g(\varphi^t(\w))=\Phi_{B}(t,\w)g(\w)=\Phi_{B_0}(t,\w)g(\w)=b_0(t,\omega)g(\varphi^t(\w)).
\end{equation}
In particular this holds between the output of $\mathcal{V}_S$ to the next input in $\mathcal{V}_S$.

\begin{claim}\label{cocycle2}
If $\varphi^t(\w),\varphi^s(\w)\notin \mathcal V_S\setminus \varphi^{\{1,2\}}(B_r)$, $b(t,\omega)$ forms a cocycle over $\varphi^t$ in the sense that $b(t+s,\w)=b(t,\varphi^s(\w))b(s,\w)$.
\end{claim}
The proof follows similarly to Claim~\ref{cocycle} taking also into account~\eqref{cocycle1emeio}.\\

Pick $\w$ in a full measure subset of points that visits infinitely often $B_r$ and for which the conclusion of Birkhoff's Ergodic theorem holds. Without loss of gene\-rality we may assume that $\w\notin \mathcal V_r\cup\mathcal V_S$.
For $t\geq0$ set
$$J_t(\w)=\#\left\{j\in\mathbb N\colon j\leq t, \, \varphi^{j}(\w)\in\varphi^{2}(B_r)\right\}.$$
Recall that
\begin{eqnarray*}
\lambda(B,\omega,  g(\omega))
	&=&\lim\limits_{t\rightarrow\infty}\frac{1}{t}\log\|\Phi_B(t,\omega) g(\omega)\|,
\end{eqnarray*}
and we may split the previous orbit in the limit by considering the time for $\varphi^t(\w)$ to enter $\mathcal{V}_S$, the time-$1$ moment crossing the flowbox $\mathcal V_S$, where  we use \eqref{ST3}, and, again, the time it takes to return to $\mathcal V_S$ and so on. For simplicity, let us define recursively
\begin{itemize}
\item[] $s_0=s_0(\w)=\min\{t\colon\varphi^t(\w)\in\varphi^1(B_r)\}$,
\item[] $\ell_0=\ell_0(\w)=s_0+1$,
\item[]  $s_n=s_0(\varphi^{\ell_{n-1}}(\w))$ and $\ell_n=s_n+1$, for $n\geq1$,
\item[] $\Delta_n=s_n-\ell_{n-1}$, for $n\geq1$,
\item[]  $\tilde\w_n=\varphi^{s_n}(\w)\in\varphi^1(B_r)$ and $\hat\w_n=\varphi^{\ell_n}(\w)\in\varphi^2(B_r)$, for $n\geq 1$.
\end{itemize}
Now, in one hand, since $B_0$ has one-point spectrum, for $\mu$-a.e. $\w$,
\begin{eqnarray}
\lambda(B_0,\w)&=&\lambda(B_0,\omega,  g(\omega))\nonumber\\
	&=&\lim\limits_{t\rightarrow\infty}\frac{1}{t}\log\|\Phi_{B_0}(t,\omega) g(\omega)\| \nonumber\\
	&\overset{\eqref{cocycle0}}{=}&\lim\limits_{t\rightarrow\infty}\frac{1}{t}\log |b_0(t,\omega)|\label{{cocycle00}}.
\end{eqnarray}
On the other hand, by Remark~\ref{estica} and~\eqref{cocycle1emeio} we have for $\tilde \w\in \varphi^1(B_r)$ that
\begin{equation}\label{RR2}
b(1,\tilde\w)g(\varphi^1(\tilde\w))=\Phi_{B}(1,\tilde \w)g(\tilde\w)\overset{\eqref{ST3}}{=}e\cdot\Phi_{B_0}(1,\tilde \w)g(\tilde\w)=e\cdot b_0(1,\tilde\w)g(\varphi^1(\tilde\w)).
\end{equation}

Without loss of generality, we can consider the following limits over the unbounded set $\{t\geq0\colon\varphi^t(\w)\in \varphi^1(B_r)\}$. From Birkhoff's Ergodic theorem we have
\begin{eqnarray*}
\lambda(B,\omega,  g(\omega))&=&\lim\limits_{t\rightarrow\infty}\frac{1}{t}\log\|\Phi_B(t,\omega) g(\omega)\|\\
	&\overset{\eqref{cocycle1eumterco}+\eqref{cocycle1emeio}}{=}&  \lim\limits_{t\rightarrow\infty}\frac{1}{t}\left(\log |b(s_0,\w)|+\sum_{j=0}^{J_t(\w)-1}\log|b(\Delta_{j+1},\hat\w_{s_j})b(1,\tilde\w_{s_j})|\right)
	\\
	&\overset{\eqref{cocycle1tresquartos}+\eqref{RR2}}{=}&  \lim\limits_{t\rightarrow\infty}\frac{1}{t}\left(\log |b_0(s_0,\w)|+\sum_{j=0}^{J_t(\w)-1}\log| b_0(\Delta_{j+1},\hat\w_{s_j})\cdot e\cdot b_0(1,\tilde\w_{s_j})|\right)
	\\
	&\overset{Claim~\ref{cocycle}}{=}&  \lim\limits_{t\rightarrow\infty}\frac{1}{t}\log|b_0(t,\w)|+\lim\limits_{t\rightarrow\infty}\frac{J_t(\w)}{t}\\
&\overset{\eqref{cocycle0}}{=}&
	\lim\limits_{t\rightarrow\infty}\frac{1}{t}\log  \|\Phi_{B_0}(t,\omega) g(\omega)\|+\lim\limits_{t\rightarrow\infty}\frac{1}{t}\int_0^t\mathbbm{1}_{\mathcal{V}_S}(\varphi^t(\w))\,dt\\
	&=&\lambda(B_0,\w,g(\w))+\mu(\mathcal{V}_S),
\end{eqnarray*}
which implies
$\lambda_1(B,\w)>\lambda(B_0,\w)$, hence $\lambda_1(B)>\lambda(B_0)$. From~\eqref{2lbaB0iguallbda1+lbda2}, we get $\lambda_1(B)>\lambda(B_0)>\lambda_2(B)$ so that $B$ has simple spectrum. Moreover, by \eqref{epsilon0}, \eqref{epsilon1} and \eqref{epsilon2} we have $\sigma_p(A,B)<\epsilon$ and Theorem~\ref{main.theorem} is now proved. $\square$

\bigskip

Clearly when considering the set $\mathcal{K}^1_0$ on Corollary~\ref{main.corollary} the equalities \eqref{Cor1} and \eqref{Cor2} become $2\lambda(B_0)=\lambda_1(B)+\lambda_2(B)=0$. Hence the conclusion this time will be that $\lambda_1(B)>0$ for $B\in \mathcal{K}^1_0$ arbitrarily $\sigma_p$-close to $A$ and also $\lambda_2(B)=-\lambda_1(B)<0$.

\vspace{1cm}
	
\textbf{Acknowledgements:} The authors were partially supported by FCT - `Funda\c{c}\~ao para a Ci\^encia e a Tecnologia', through Centro de Matem\'atica e Aplica\c{c}\~oes (CMA-UBI), Universidade da Beira Interior, project UIDB/MAT/00212/2020. MB was partially supported by the Project `Means and Extremes in Dynamical Systems' (PTDC/MAT-PUR/4048/2021).  MB also like to thank CMUP for providing the necessary conditions in which this work was developed.

\end{document}